\documentclass[a4paper]{amsart}

\usepackage{xcolor}
\usepackage[utf8]{inputenc}
\usepackage[T1]{fontenc}
\usepackage{lmodern}
\usepackage{amssymb,amsxtra}
\usepackage{nicefrac,mathtools}
\usepackage[lite]{amsrefs}
\newcommand*{\MRref}[2]{ \href{http://www.ams.org/mathscinet-getitem?mr=#1}{MR \textbf{#1}}}

\newcommand*{\arxiv}[1]{\href{http://www.arxiv.org/abs/#1}{arXiv: #1}}

\renewcommand{\PrintDOI}[1]{\href{http://dx.doi.org/\detokenize{#1}}{doi: \detokenize{#1}}}

\BibSpec{book}{%
  +{}  {\PrintPrimary}                {transition}
  +{,} { \textit}                     {title}
  +{.} { }                            {part}
  +{:} { \textit}                     {subtitle}
  +{,} { \PrintEdition}               {edition}
  +{}  { \PrintEditorsB}              {editor}
  +{,} { \PrintTranslatorsC}          {translator}
  +{,} { \PrintContributions}         {contribution}
  +{,} { }                            {series}
  +{,} { \voltext}                    {volume}
  +{,} { }                            {publisher}
  +{,} { }                            {organization}
  +{,} { }                            {address}
  +{,} { \PrintDateB}                 {date}
  +{,} { }                            {status}
  +{}  { \parenthesize}               {language}
  +{}  { \PrintTranslation}           {translation}
  +{;} { \PrintReprint}               {reprint}
  +{.} { }                            {note}
  +{.} {}                             {transition}
  +{} { \PrintDOI}                   {doi}
  +{} { available at \url}            {eprint}
  +{}  {\SentenceSpace \PrintReviews} {review}
}

\usepackage{enumitem} 
\setlist[enumerate,1]{label=\textup{(\arabic*)}}

\usepackage{tikz}
\usetikzlibrary{matrix,calc,arrows,decorations.pathmorphing}
\tikzset{node distance=2cm, auto}

\tikzset{cd/.style=matrix of math nodes,row sep=2em,column sep=2em, text height=1.5ex, text depth=0.5ex}
\tikzset{cdar/.style=->,auto}
\tikzset{mid/.style={anchor=mid}} 
\tikzset{narrowfill/.style={inner sep=1pt, fill=white}}

\usepackage{microtype}
\usepackage[pdftitle={C*-algebras that are not groupoid C*-algebras},
  pdfauthor={Alcides Buss, Aidan Sims},
  pdfsubject={Mathematics}
]{hyperref}

\numberwithin{equation}{section}
\theoremstyle{plain}
\newtheorem{theorem}[equation]{Theorem}

\newtheorem{corollary}[equation]{Corollary}
\theoremstyle{definition}

\theoremstyle{remark}
\newtheorem{remark}[equation]{Remark}

\newcommand*{\nb}{\nobreakdash}
\newcommand*{\Star}{\(^*\)\nobreakdash-}

\newcommand*{\C}{\mathbb C}
\newcommand*{\Z}{\mathbb Z}

\newcommand*{\N}{\mathbb N}

\newcommand*{\Torus}{\mathbb T}
\newcommand*{\Bound}{\mathbb B}
\newcommand*{\Comp}{\mathbb K}
\newcommand*{\red}{\textup{r}} 
\newcommand*{\op}{\textup{op}} 
\newcommand*{\oo}{\textup{o}} 
\newcommand*{\univ}{\textup{u}} 
\newcommand*{\A}{\mathcal A}
\newcommand*{\E}{\mathcal E}
\newcommand*{\Rr}{\mathcal R}
\newcommand*{\Ss}{\mathcal S}

\newcommand*{\Sect}{\mathfrak C_c}

\newcommand*{\cont}{\textup C}
\newcommand*{\contz}{\cont_0}

\newcommand*{\defeq}{\mathrel{\vcentcolon=}}
\newcommand*{\congto}{\xrightarrow\sim}

\newcommand*{\braket}[2]{\langle#1{\mid}#2\rangle}

\newcommand*{\s}{s}
\newcommand*{\rg}{r}

\newcommand*{\into}{\hookrightarrow}
\newcommand*{\onto}{\twoheadrightarrow}

\newcommand*{\cstar}{\texorpdfstring{$C^*$\nobreakdash-\hspace{0pt}}{*-}}

\newcommand*{\dd}{\,\mathrm d}

\newcommand*{\sbe}{\subseteq} 

\begin{document}
\title{Opposite algebras of groupoid $C^*$-algebras}

\author{Alcides Buss}
\email{alcides@mtm.ufsc.br}
\address{Departamento de Matem\'atica\\
 Universidade Federal de Santa Catarina\\
 88.040-900 Florian\'opolis-SC\\
 Brazil}

\author{Aidan Sims}
\email{asims@uow.edu.au}
\address{School of Mathematics and Applied Statistics \\
The University of Wollongong\\
NSW  2522\\ Australia}

\begin{abstract}
We show that every groupoid \cstar{}algebra is isomorphic to its opposite, and deduce
that there exist \cstar{}algebras that are not stably isomorphic to groupoid
\cstar{}algebras, though many of them are stably isomorphic to twisted groupoid
\cstar{}algebras. We also prove that the opposite algebra of a section algebra of a Fell
bundle over a groupoid is isomorphic to the section algebra of a natural opposite bundle.
\end{abstract}
\thanks{The first author is supported by CNPq (Brazil). The second author was supported by the Australian Research Council grant DP150101598.
We would like to thank Chris Phillips and Ilijas Farah for helpful discussions and references to examples of non-self-opposite C*-algebras.}
\maketitle


\section{Introduction}
\label{sec:introduction}

Groupoids are among the most widely used models for operator algebras. It is therefore a
basic question whether a given \cstar{}algebra $A$ can be realised as $C^*(G)$ or
$C^*_\red(G)$ for some locally compact topological groupoid $G$. Many classes of
\cstar{}algebras have groupoid models: for example, graph \cstar{}algebras and
higher-rank graph \cstar{}algebras, \cstar{}algebras of actions of inverse semigroups,
and \cstar{}algebras associated to foliations. Moreover, it follows from the main results
in \cite{Exel-Pardo:Self-similar} that every Kirchberg \cstar{}algebra (that is, every
separable, simple, nuclear, purely infinite \cstar{}algebra) has an \'etale groupoid
model.

We show in this paper that not every \cstar{}algebra has a groupoid model. We achieve
this by showing that all groupoid \cstar{}algebras are self-opposite in the sense that
they are isomorphic to their opposite \cstar{}algebras.

Several examples of non-self-opposite \cstar{}algebras are already known. The first,
produced by Connes \cite{Connes:Factor}, is a non-self-opposite von Neumann factor.
Later, examples of non-self-opposite separable \cstar{}algebras were found by Phillips in
\cite{Phillips:Continuous-trace}. All of Phillips' examples are continuous-trace
\cstar{}algebras, hence nuclear. Simple and separable non-self-opposite $C^*$-algebras
are constructed in \cite{Phillips-Viola:simple_exact, Phillips:simple}; these examples
are non-nuclear, though the one in \cite{Phillips-Viola:simple_exact} is exact. It
remains open whether there exists a simple, separable and nuclear non-self-opposite
\cstar{}algebra~\cite{Phillips-Dadarlat-Hirshberg:Simple_nuclear}. This is related to
Elliott's conjecture (see \cite{Rordam:Classification_survey}) because the Elliott
invariant (essentially $K$-groups) used in the conjecture cannot distinguish a
\cstar{}algebra $A$ from its opposite $A^\op$.

Although our result implies the existence of \cstar{}algebras with no groupoid model, it
is still possible that such \cstar{}algebras can be realised as twisted groupoid
\cstar{}algebras. That is, they could be isomorphic to $C^*(G,\Sigma)$ or
$C^*_\red(G,\Sigma)$, for some twist $\Sigma$ over a groupoid $G$. A twist over $G$ is
essentially the same thing as a Fell line bundle $L$ over $G$, and $C^*(G, \Sigma)$ and
$C^*_{\red}(G,\Sigma)$ are then the corresponding full and reduced cross-sectional
\cstar{}algebras $C^*(G,L)$ and $C^*_{\red}(G,L)$. Renault proves in
\cite{Renault:Cartan.Subalgebras} that every \cstar{}algebra $A$ admitting a Cartan
subalgebra $\contz(X)\subseteq A$ is isomorphic to $C^*_\red(G,\Sigma)$ for some (second
countable, locally compact Hausdorff) \'etale essentially principal groupoid $G$ with
$G^0=X$ and some twist $\Sigma$ on $G$; furthermore, the pair $(G,\Sigma)$ is uniquely
determined by the Cartan pair $(A,\contz(X))$.

Kumjian, an Huef and Sims proved in \cite{Huef-Kumjian-Sims:Dixmier-Douady} that every
Fell \cstar{}algebra (in particular, every continuous-trace \cstar{}algebra) is Morita
equivalent to a \cstar{}algebra with a diagonal subalgebra in the sense of Kumjian
\cite{Kumjian:Diagonals}. These diagonal subalgebras are exactly the Cartan subalgebras
(in the sense of Renault) with the \emph{unique extension property}: every pure state of
the Cartan subalgebra $\contz(X)$ extends uniquely to $A$. The corresponding twist
$(G,\Sigma)$ that describes $(A,\contz(X))$ is over a principal, not just essentially
principal, groupoid $G$. After stabilisation, these results imply that all
continuous-trace \cstar{}algebras have a twisted groupoid model---including the examples
of Phillips in \cite{Phillips:Continuous-trace} that do not admit untwisted groupoid
models. The point is that the opposite algebra of $C^*(G, \Sigma)$ arises as the
$C^*$-algebra $C^*(G, \overline{\Sigma})$ of the conjugate twist, and this corresponds to
taking the negative of the associated Dixmier--Douady invariant.

We elucidate the above phenomenon by describing the opposite \cstar{}algebras
$C^*(G,\A)^\op$ and $C^*_\red(G,\A)^\op$ of the cross-sectional algebras of arbitrary
Fell bundles $\A$ over locally compact groupoids. Specifically, given a Fell bundle $\A$
over $G$, we construct an appropriate opposite bundle $\A^\oo$ over $G$, and prove that
$C^*(G,\A)^\op\cong C^*(G,\A^\oo)$. This can also be described in terms of the conjugate
Fell bundle $\bar\A$. In the special case of a Fell line bundle $L$ (that is, a twist
over $G$), this corresponds to the conjugate line bundle. When $L$ is the trivial line
bundle, $\bar{L} = L$, and $C^*_{\red}(G; L)$ and $C^*(G; L)$ coincide with
$C^*_{\red}(G)$ and $C^*(G)$, so we recover our earlier result as a special case.

For a Fell bundle associated to an action $\alpha$ of a locally compact group $G$ on a
\cstar{}algebra $A$, our result is equivalent to the statement that the opposite
\cstar{}algebras of the full and reduced crossed products $A\rtimes_\alpha G$ and
$A\rtimes_{\alpha,\red} G$ are isomorphic to $A^\op\rtimes_{\alpha^\op}G$ and
$A^\op\rtimes_{\alpha^\op,\red}G$ (where $\alpha^\op$ is the action of $G$ on $A^\op$
determined by $\alpha$ upon identifying $A$ and $A^\op$ as linear spaces); this was
proved for full crossed products in \cite{Phillips-Dadarlat-Hirshberg:Simple_nuclear}.

\section{Groupoid C*-algebras and their opposites}

In this section we show that the full and reduced C*-algebras of a locally compact,
locally Hausdorff groupoid with Haar system are self-opposite. We first briefly recall
how these \cstar{}algebras are defined.

Let $G$ be a locally compact and locally Hausdorff groupoid with Hausdorff unit space
$G^0$ and a (continuous) left invariant Haar system $\lambda=\{\lambda^x\}_{x\in G^0}$.
Let $\Sect(G,\lambda)$ be the \Star{}algebra of compactly supported, quasi-continuous
sections, that is, the linear span of continuous functions  with compact support $f\colon
U\to \C$ on open Hausdorff subsets $U\sbe G$. These functions are extended by zero off
$U$ and hence viewed as functions $G\to \C$. The continuity of $\lambda$ means that every
such function is mapped to a continuous function $\lambda(f)\colon G^0\to \C$ via
$\lambda(f)(x)\defeq \int_G f(g)\dd\lambda^x(g)$. By definition, $\lambda^x$ is a Radon
measure on $G$ with support $G^x\defeq \rg^{-1}(x)$ for all $x\in G^0$. Recall that the
product and involution on $\Sect(G,\lambda)$ are defined by
\[
(f_1*f_2)(g)\defeq \int_G f_1(h)f_2(h^{-1}g)\dd\lambda^{\rg(g)}(h)
    \quad\text{ and }\quad
f^*(g)\defeq \overline{f(g^{-1})}.
\]
Under these operations and the inductive-limit topology, $\Sect(G,\lambda)$ is a
topological \Star{}algebra. The $I$-norm on $\Sect(G,\lambda)$ is defined by
\[
    \|f\|_I\defeq \max\{\|\lambda(|f|)\|_\infty,\|\lambda(|f^*|)\|_\infty\}.
\]
The $L^1$-Banach \Star{}algebra of $G$ is the completion of $\Sect(G,\lambda)$ with
respect to $\|\cdot\|_I$; we denote it by $L^1_I(G,\lambda)$. The full \cstar{}algebra of
$G$ is the universal enveloping \cstar{}algebra of $L^1_I(G,\lambda)$; in other words, it
is the \cstar{}completion of $\Sect(G,\lambda)$ with respect to the maximum
$\|\cdot\|_I$-bounded \cstar{}norm:
\[
    \|f\|_\univ\defeq \sup\{\|\pi(f)\| : \pi \mbox{ is an $I$-norm decreasing \Star{}representation of }\Sect(G,\lambda)\}.
\]
The \emph{regular representations} of $(G,\lambda)$ are the representations
\[
\pi_x\colon \Sect(G,\lambda)\to \Bound(L^2(G_x,\lambda_x)),\quad x \in G^{(0)}
\]
given by $\pi_x(f)\xi(g)\defeq (f*\xi)(g)=\int_G f(gh)\xi(h^{-1})\dd\lambda^x(h)$. Here
$\lambda_x$ is the image of $\lambda^x$ under the inversion map $G\to G$, $g\mapsto
g^{-1}$; so it is a measure with support $G_x=\s^{-1}(x)$. The system of measures
$(\lambda_x)_{x\in G^0}$ is a right invariant Haar system on $G$.

The regular representations of $G$ give rise to a $\|\cdot\|_I$-bounded \cstar{}norm
called the \emph{reduced \cstar{}norm}:
\[
    \|f\|_\red\defeq \sup_{x\in G^0}\|\pi_x(f)\|.
\]
The \emph{reduced \cstar{}algebra} of $G$ is the completion of $\Sect(G,\lambda)$ with
respect to $\|\cdot\|_\red$. It is denoted by $C^*_\red(G,\lambda)$.

Given a groupoid $G$, we write $G^\op$ for the opposite groupoid, equal to $G$ as a
topological space, but with $(G^\op)^{(2)} = \{(h,g) : (g,h) \in G^{(2)}\}$ and
composition given by $h \cdot_\op g = gh$. We write $\lambda^\op$ for the Haar system on
$G^\op$ defined as the image of $\lambda$ under the inversion map regarded as a
homeomorphism of $G$ onto $G^\op$.

\begin{theorem}\label{theo:groupoid-algebras-symmetric}
Let $(G,\lambda)$ be a locally compact, locally Hausdorff groupoid with Haar system. The
inversion map $g\mapsto g^{-1}$ defines an isomorphism $(G,\lambda)\cong
(G^\op,\lambda^\op)$ of topological groupoids with Haar systems.  Given $f\in
\Sect(G,\lambda)$, define $f^\op : G \to \C$ by $f^\op(g)\defeq f(g^{-1})$. Then
$f\mapsto f^\op$ is an isomorphism $\Sect(G,\lambda)\congto \Sect(G,\lambda)^\op$ of
topological \Star{}algebras. This isomorphism extends to a Banach \Star{}algebra
isomorphism $L^1_I(G,\lambda)\congto L^1_I(G,\lambda)^\op$ and to \cstar{}algebra
isomorphisms $C^*(G,\lambda)\congto C^*(G,\lambda)^\op$ and $C^*_\red(G,\lambda)\congto
C^*_\red(G,\lambda)^\op$.
\end{theorem}
\begin{proof}
The map $g\mapsto g^{-1}$ is a homeomorphism $G \to G^\op$ and satisfies $g^{-1}
\cdot_\op h^{-1} = h^{-1} g^{-1} = (gh)^{-1}$, so it is an isomorphism $G\congto G^\op$
of topological groupoids. The range (resp. source) map of $G^\op$ is the source (resp.
range) map of $G$, and the inversion map sends the left invariant Haar system
$\lambda=(\lambda^x)_{x\in G^0}$ on $G$ to the right invariant Haar system
$(\lambda_x)_{x\in G^0}$, which is precisely $\lambda^\op$. This yields the isomorphism
$(G,\lambda)\cong (G^\op,\lambda^\op)$. The map $f\mapsto f^\op$ is a linear involution
(in particular, a bijection) which is clearly a homeomorphism with respect to the
inductive-limit topology. It is also clearly isometric for the $\|\cdot\|_I$-norms on
$\Sect(G,\lambda) $ and $\Sect(G^\op,\lambda^\op)$. So to prove that it is topological
\Star{}algebra isomorphism $\Sect(G,\lambda)\cong \Sect(G,\lambda)^\op$ and extends to
isomorphisms $L^1_I(G,\lambda)\cong L^1_I(G,\lambda)^\op$ and $C^*(G,\lambda)\cong
C^*(G,\lambda)^\op$, it suffices to show that $f\mapsto f^\op$ is a \Star{}homomorphism.
For $f\in \Sect(G,\lambda)$,
\[
    (f^\op)^*(g)=\overline{f^\op(g^{-1})}=\overline{f(g)}=f^*(g^{-1})=(f^*)^\op(g).
\]
So $f\mapsto f^\op$ preserves involution. If $f_1,f_2\in \Sect(G,\lambda)$, then
\begin{equation}\label{eq:conv1}
(f_1*f_2)^\op(g)=(f_1*f_2)(g^{-1})=\int_G f_1(h)f_2(h^{-1}g^{-1})\dd\lambda^{\s(g)}(h),
\end{equation}
while
\begin{equation}\label{eq:conv2}
\begin{split}
(f_2^\op*f_1^\op)(g)
    &= \int_G f_2^\op(h)f_1(h^{-1}g)\dd\lambda^{\rg(g)}(h)\\
    &= \int_Gf_1(g^{-1}h)f_2(h^{-1})\dd\lambda^{\rg(g)}(h).
\end{split}
\end{equation}
Making the change of variables $h\mapsto gh$ and applying left invariance of $\lambda$
shows that~\eqref{eq:conv1} and~\eqref{eq:conv2} are equal.

To prove that $C^*_\red(G,\lambda)\cong C^*_\red(G,\lambda)^\op$, observe that the map
$f\mapsto f^\op$ gives an isomorphism $L^2(G_x,\lambda_x)\cong
L^2(G^x,\lambda^x)=L^2(G_x^\op,\lambda_x^\op)$ which induces a unitary equivalence
between the regular representations $\pi_x\colon \Sect(G,\lambda)\to
\Bound(L^2(G_x,\lambda_x)$ and $\pi_x^\op\colon \Sect(G^\op,\lambda^\op)\to
\Bound(L^2(G^\op_x,\lambda^\op_x)$. This yields the equality $\|f\|_\red=\|f^\op\|_\red$
which shows that $f \mapsto f^\op$ extends to an isomorphism $C^*_\red(G,\lambda)\cong
C^*_\red(G^\op,\lambda^\op)$.
\end{proof}

\begin{remark}
Similar arguments to those above show that the identity map on $G$, regarded as an
anti-multiplicative homeomorphism from $G$ to $G^\op$, induces (by composition) an
anti-multiplicative linear isomorphism $\Sect(G,\lambda) \cong \Sect(G^\op,\lambda^\op)$,
and therefore a topological $^*$-algebra isomorphism $\Sect(G,\lambda)^\op\cong
\Sect(G^\op,\lambda^\op)$. This latter extends to isomorphisms $L^1_I(G,\lambda)^\op\cong
L^1_I(G^\op,\lambda^\op)$, $C^*(G,\lambda)^\op\cong C^*(G^\op,\lambda^\op)$, and
$C^*_\red(G,\lambda)^\op\cong C^*_\red(G^\op,\lambda^\op)$.

Another way to prove Theorem~\ref{theo:groupoid-algebras-symmetric} is to work with
\emph{conjugate algebras}. If $A$ is a \Star{}algebra, its conjugate \Star{}algebra $\bar
A$ is the conjugate vector space of $A$ endowed with the same algebraic operations as
$A$. Involution, $a \mapsto a^*$ is then a linear anti-multiplicative isomorphism $A \to
\bar A$ and therefore an isomorphism $A^\op\cong \bar A$. We have $\Sect(G,\lambda)\cong
\overline{\Sect(G,\lambda)}$ via $\xi\mapsto \bar\xi$ and this extends to isomorphisms
$L^1_I(G,\lambda)\cong \overline{L^1_I(G,\lambda)}$, $C^*(G,\lambda)\cong
\overline{C^*(G,\lambda)}$ and $C^*_\red(G,\lambda)\cong \overline{C^*_\red(G,\lambda)}$.
\end{remark}

\begin{corollary}
There are (nuclear, separable) \cstar{}algebras that are not isomorphic to either
$C^*(G,\lambda)$ or $C^*_\red(G,\lambda)$ for any locally compact, locally Hausdorff
groupoid with Haar system.
\end{corollary}
\begin{proof}
It is known that there are examples of nuclear and separable \cstar{}algebras that are
not self-opposite \cite{Phillips:Continuous-trace,
Phillips-Dadarlat-Hirshberg:Simple_nuclear}.
\end{proof}

Let us say that a \Star{}algebra $A$ is \emph{self-opposite} if $A \cong A^\op$. Our main
result says that given a topological groupoid with Haar system $(G,\lambda)$, the
\Star{}algebras $\Sect(G, \lambda)$, $L^1_I(G, \lambda)$, $C^*(G, \lambda)$ and
$C^*_{\red}(G,\lambda)$ are all self-opposite. Both the minimal and the maximal tensor
product of self-opposite \cstar{}algebras are again self-opposite because $(A\otimes
B)^\op\cong A^\op\otimes B^\op$.

Let $\Comp$ denote the \cstar{}algebra of compact operators on a separable, infinite
dimensional Hilbert space; writing $\Rr$ for the equivalence relation $\N \times \N$
regarded as a discrete principal groupoid, we have $\Comp \cong C^*(\Rr) =
C^*_{\red}(\Rr)$. Hence the preceding paragraph shows that every self-opposite
\cstar{}algebra is also stably self-opposite. The converse fails in general: Phillips
constructs in \cite{Phillips:Continuous-trace} examples of (separable, continuous-trace)
non-self-opposite \cstar{}algebras which are stably self-opposite. But Phillips also
constructs examples of (separable, continuous-trace) \cstar{}algebras that are not stably
self-opposite. This yields the following:

\begin{corollary}\label{cor:no-groupoid-C-algebra-continuous-trace}
There are separable continuous-trace \cstar{}algebras that are not stably isomorphic to
any groupoid \cstar{}algebra.
\end{corollary}

\begin{remark}
By the Brown--Green--Rieffel theorem \cite{Brown-Green-Rieffel:Stable},
Corollary~\ref{cor:no-groupoid-C-algebra-continuous-trace} implies that there exist
separable \cstar{}algebras that are not Morita equivalent to a separable (or even
$\sigma$-unital) groupoid \cstar{}algebra. However, it is unclear whether these examples
could be Morita equivalent to a non-$\sigma$-unital groupoid \cstar{}algebra.
\end{remark}

In \cite{Farah:UHF-non-self-opposite}, in the framework of ZFC enriched with Jensen's
diamond principle (a strengthening of the continuum hypothesis), Farah and Hirshberg
construct examples of non-separable approximately matricial algebras (uncountable direct
limits of the CAR algebra) that are non-self-opposite, so we can also state:

\begin{corollary}
It is consistent with ZFC that there are non-separable approximately matricial (so
simple, nuclear) \cstar{}algebras that are not isomorphic to a groupoid \cstar{}algebra.
\end{corollary}

Recall that the ordinary \emph{separable} AF-algebras admit groupoid models: it is even
known that they are always crossed products for a partial action of the integers, see
\cite{Exel:AF}.

By \cite{Huef-Kumjian-Sims:Dixmier-Douady}*{Theorem~6.6(1)}, every separable
continuous-trace \cstar{}algebra (indeed, every Fell algebra) is Morita equivalent to a
separable \cstar{}algebra with a diagonal subalgebra in the sense of
Kumjian~\cite{Kumjian:Diagonals}. Kumjian shows in \cite{Kumjian:Diagonals} that
\cstar{}algebras containing diagonals are, up to isomorphism, the \cstar{}algebras
obtained from twists on \'etale principal groupoids. More precisely, this means a pair
$(G,\Sigma)$ consisting of a (second countable) locally compact Hausdorff \'etale
groupoid $G$ and another (locally compact Hausdorff, second countable) topological
groupoid $\Sigma$ that fits into a central groupoid extension of the form
\[
    \Torus\times G^0\into \Sigma\onto G,
\]
where $\Torus$ denotes the circle group, and $\Torus\times G^0$ is viewed as a (trivial)
group bundle, and hence as a topological groupoid. To a twisted groupoid $(G,\Sigma)$ one
can assign a full \cstar{}algebra $C^*(G,\Sigma)$ and a reduced \cstar{}algebra
$C^*_\red(G,\Sigma)$, and then every separable \cstar{}algebra containing a diagonal
subalgebra has the form $C^*_\red(G,\Sigma)$ for some twist $\Sigma$ over a principal
groupoid $G$. Moreover, the pair $(G, \Sigma)$ is unique, up to isomorphism of twisted
groupoids. This follows from the more general result, proved by Renault
in~\cite{Renault:Cartan.Subalgebras}, that isomorphism classes of Cartan subalgebras
correspond bijectively to isomorphism classes of twisted essentially principal \'etale
groupoids (meaning twisted groupoids where $G$ is not necessarily principal, but only
essentially principal; see \cite{Renault:Cartan.Subalgebras} for details). Using these
results, we arrive at the following consequence:

\begin{corollary}
There are separable stable continuous-trace \cstar{}algebras that are not isomorphic to
any groupoid \cstar{}algebra but which are isomorphic to the reduced \cstar{}algebra of a
twisted principal \'etale groupoid.
\end{corollary}
\begin{proof}
Let $A$ be a separable continuous-trace \cstar{}algebra which is not stably isomorphic to
any groupoid \cstar{}algebra as in
Corollary~\ref{cor:no-groupoid-C-algebra-continuous-trace}. Let $B\defeq A\otimes\Comp$
be the stabilisation of $A$. Then $B$ is a separable stable continuous-trace
\cstar{}algebra which is not isomorphic to any groupoid \cstar{}algebra.
By~\cite{Huef-Kumjian-Sims:Dixmier-Douady}*{Theorem~6.(1)} $A$ is Morita equivalent to
$C^*_\red(G,\Sigma)$, for some twisted principal \'etale groupoid $(G,\Sigma)$. It
follows from the Brown--Green--Rieffel theorem that $A\cong
C^*_\red(G,\Sigma)\otimes\Comp$. To finish the proof we observe that, again writing $\Rr$
for the discrete equivalence relation $\N \times \N$, we have $C^*_\red(G,\Sigma)\otimes
\Comp \cong C^*_\red(G \times \Rr, \Sigma \times \Rr)$.
\end{proof}

\section{Section \cstar{}algebras of Fell bundles and their opposites}

Let $G$ be a locally compact and locally Hausdorff groupoid endowed with a continuous
Haar system $\lambda$, which we fix throughout the rest of the section. In this section
we generalise our previous result and describe the opposite \cstar{}algebras of the
section \cstar{}algebras of a Fell bundles over $G$. Our result generalises the
observation in \cite{Phillips-Dadarlat-Hirshberg:Simple_nuclear} that $(A\rtimes_\alpha
G)^\op\cong A^\op\rtimes_{\alpha^\op}G$ for any action $\alpha$ of a locally compact
group $G$ on a \cstar{}algebra $A$.

Fell bundles over topological groupoids are defined in \cite{Kumjian:Fell_bundles}. Only
Hausdorff groupoids are considered there, but the same definition makes sense for locally
Hausdorff groupoids. A Fell bundle over $G$ consists of an upper semicontinuous Banach
bundle $\A$ over $G$ endowed with \emph{multiplications} $\A_g\times\A_h\to \A_{gh}$,
$(a,b)\mapsto a\cdot b$, for every composable pair $(g,h)\in G^2$ and \emph{involutions}
$\A_g\to\A_{g^{-1}}$, $a\mapsto a^*$, for every $g\in G$. These operations are required
to be continuous (with respect to the given topology on $\A$) and satisfy algebraic
conditions similar to those in the definition of a \cstar{}algebra.

We next recall, briefly, how to define the full and reduced \cstar{}algebras of a Fell
bundle. Consider the space $\Sect(G,\A)$ of compactly supported continuous sections
$\xi\colon U\to \A$ defined on open Hausdorff subspaces $U\sbe G$ and extended by zero
outside $U$ and hence viewed as sections $\xi\colon G\to \A$. The continuity of the
algebraic operations on $\A$ implies that for $\xi, \eta \in \Sect(G,\A)$, the formulas
\[
(\xi*\eta)(g)\defeq \int_G\xi(h)\cdot \eta(h^{-1}g)\dd\lambda^{\rg(g)}(h),\quad\text{ and }\quad \xi^*(g)\defeq \xi(g^{-1})^*.
\]
define elements $\xi*\eta, \xi^* \in \Sect(G,\A)$ and so determine a convolution product
$*$ and an involution $^*$ on $\Sect(G,\A)$. Under these operations, $\Sect(G, \A)$ is a
\Star{}algebra; and indeed, a topological $^*$-algebra in the inductive-limit topology.

Since the norm function on $\A$ is upper semicontinuous, the function $g\mapsto
\|\xi(g)\|$ from $G$ to $[0,\infty)$ is upper semicontinuous and hence measurable. So we
can define the $I$-norm on $\Sect(G, \A)$ by
\[
\|\xi\|_I \defeq \sup_{x \in G^{(0)}} \max\Big\{ \int_{G^x} |\xi(g)|\dd\lambda^x(g), \int_{G^x} |\xi^*(g)|\dd\lambda^x(g)\Big\}.
\]
The $L^1$-Banach algebra of $\A$, denoted $L^1_I(G,\A)$, is defined as the completion of
$\Sect(G,\A)$ with respect to $\|\cdot\|_I$. The full \cstar{}algebra $C^*(G,\A)$ of $\A$
is defined as the universal enveloping \cstar{}algebra of $L^1_I(G,\A)$: the completion
of $\Sect(G,\A)$ with respect to the \cstar{}norm
\[
\|\xi\|_\univ\defeq\sup\{\|\pi(\xi)\| : \text{$\pi$ is an $I$-norm decreasing $^*$-representation of $\Sect(G, \A)$}\}.
\]
That this is indeed a norm on $\Sect(G,\A)$, and not just a seminorm, follows from the
existence of the following regular representations.

For each $x \in G^{(0)}$, let $L^2(G_x,\A)$ be the Hilbert $\A_x$-module completion of
the space $\Sect(G_x,\A)$ of quasi-continuous sections $G_x\to \A$ with respect to the
norm induced by the $\A_x$-valued inner product
\[
\braket{\xi}{\eta}_{\A_x}\defeq \int_G\xi(h)^*\eta(h)\dd\lambda_x(h)=\int_G\xi(h^{-1})^*\eta(h^{-1})\dd\lambda^x(h).
\]
Then for each $x \in G^{(0)}$, the regular representation $\pi_x : \Sect(G, \A) \to
\Bound(L^2(G_x,\A))$ is defined by
\[
\big(\pi_x(\xi)\eta\big)(g)\defeq\int_G\xi(gh)\eta(h^{-1})\dd\lambda^x(h)=\int_G\xi(gh^{-1})\eta(h)\dd\lambda_x(h),
\]
for all $\xi\in \Sect(G,\A)$, $\eta\in \Sect(G_x,\A)$ and $g\in G_x$. The reduced
\cstar{}norm on $\Sect(G,\A)$ is defined by
\[
\|\xi\|_\red\defeq \sup_{x\in G^0}\|\pi_x(\xi)\|.
\]
This is, indeed, a norm: if $\pi_x(\xi)=0$ then $(\xi*\eta)(g)=0$ for all $\eta\in
\Sect(G_x,\A)$ and $g\in G_x$; so $\xi*\xi^*(x)=\int_G\xi(h)\xi(h)^*\dd\lambda^x(h)=0$
for all $x \in G^{(0)}$, forcing $\xi|_{G^x}=0$ for all $x$. A standard computation shows
that $\|\xi\|_\red\leq \|\xi\|_I$. Therefore $\|\cdot\|_\univ$ is also a \cstar{}norm and
$\|\cdot \|_\red \leq \|\cdot \|_\univ$. The completion of $\Sect(G,\A)$ with respect to
$\|\cdot\|_\red$ is the \emph{reduced section \cstar{}algebra} of $\A$, and is denoted by
$C^*_\red(G,\A)$.

Our goal here is to describe the opposite \cstar{}algebras $C^*(G,\A)^\op$ and
$C^*_\red(G,\A)^\op$. We show that $C^*(G,\A)^\op)\cong C^*(G,\A^\oo)$, for an
appropriate opposite Fell bundle $\A^\oo$ over $G$ associated to $\A$. It is more natural
to first define an opposite Fell bundle $\A^\op$ over the opposite groupoid $G^\op$ and
then later use the canonical anti-isomorphism $G\cong G^\op$ induced by the inversion map
to obtain the desired Fell bundle $\A^\oo$ over $G$.

The opposite Fell bundle $\A^\op$ over $G^\op$ is defined as follows. As a Banach bundle,
$\A^\op$ does not differ from $\A$. In particular, the fibres are equal, $\A^\op_g=\A_g$
for all $g\in G$, and also the topology on $\A^\op$ is equal to that on $\A$. Moreover,
$\A^\op$ is also endowed with the same involution as $\A$, which makes sense because $G$
and $G^\op$ carry the same inversion map. The only thing that changes in $\A^\op$ is the
multiplication: given $g,h\in G^\op$ the condition $\s^\op(g)=\rg^\op(h)$ means
$\rg(g)=\s(h)$, so we can use the multiplication map $\mu_{h,g}\colon \A_h\times\A_g\to
\A_{hg}$ and define the multiplication maps
\[
\mu^\op\colon \A_g^\op\times \A_h^\op=\A_g\times \A_h\to \A_{g\cdot_\op h}^\op=\A_{hg}\quad\mbox{by } \mu^\op(a,b)\defeq \mu(b,a).
\]
In other words, $a\cdot_\op b\defeq b\cdot a$ if we use $\cdot$ and $\cdot_\op$ to denote
the multiplications on $\A$ and $\A^\op$, respectively. It is straightforward to see that
$\A^\op$ is indeed a Fell bundle over $G^\op$. Now we use the anti-isomorphism
$G^\op\cong G$ induced by the inversion map $g\mapsto g^{-1}$ to form the pullback Fell
bundle of $\A^\op$. In other words, $\A^\oo$ is a Fell bundle over $G$ with fibres
$\A^\oo_g=\A_{g^{-1}}$ and the topology induced by the sections $\xi^\oo(g)\defeq
\xi(g^{-1})$ for $\xi\colon U\to \A$ a continuous section defined on a Hausdorff open
subset $U\sbe G$. The involution map $\A^\oo_g\to \A^\oo_{g^{-1}}$ is the involution map
$\A_{g^{-1}}\to \A_g$ from $\A$ and the multiplication map $\A^\oo_g\times \A^\oo_h\to
\A^\oo_{gh}$ is given by $(a,b)\mapsto b\cdot a$ for all $b\in \A^\oo_g=\A_{g^{-1}}$,
$b\in \A^\oo_h=\A_{h^{-1}}$ and $g,h\in G$ with $\s(g)=\rg(h)$.

\begin{theorem}\label{thm-fell-op-iso}
Let $\A$ be a Fell bundle over a locally compact, locally Hausdorff groupoid with Haar
system $(G,\lambda)$, and consider the Fell bundle $\A^\oo$ over $(G,\lambda)$ described
above. The map $\xi\mapsto \xi^\oo$ defined by $\xi^\oo(g)\defeq \xi(g^{-1})$ gives an
isomorphism of topological \Star{}algebras $\Sect(G,\A)^\op\congto \Sect(G,\A^\oo)$.
Moreover, this isomorphism extends to an isomorphism of Banach \Star{}algebras
$L^1_I(G,\A)^\op\congto L^1_I(G,\A^\oo)$ and \cstar{}algebras $C^*(G,\A)^\op\congto
C^*(G,\A^\oo)$ and $C^*_\red(G,\A)^\op\congto C^*_\red(G,\A^\oo)$.
\end{theorem}
\begin{proof}
We prove the equivalent assertion that $\Sect(G,\A)^\op \cong \Sect(G^\op,\A^\op)$ via
the canonical linear isomorphism $\Sect(G,\A)\ni\xi\mapsto \xi^\op\defeq \xi\in
\Sect(G^\op,\A^\op)$, and that this isomorphism extends to isomorphisms
\begin{gather*}
L^1_I(G,\A)^\op\congto L^1_I(G^\op,\A^\op), \quad
    C^*(G,\A)^\op\congto C^*(G^\op,\A^\op),\quad\text{and}\\
    C^*_\red(G,\A)^\op\congto C^*_\red(G^\op,\A^\op).
\end{gather*}

Since the topologies on $\A$ and $\A^\op$ are the same, the map $\xi\mapsto \xi^\op$ is
clearly a linear bijection $\Sect(G,\A)\to \Sect(G^\op,\A^\op)$ which is a homeomorphism
with respect to the inductive-limit topologies. Also, this map preserves the involution,
that is, $(\xi^\op)^*=\xi^*$ on $\Sect(G,\A)$ (which is the same as the involution on
$\Sect(G,\A)^\op$), and on $\Sect(G^\op,\A^\op)$ because the involutions on $\A$ and on
$\A^\op$ are the same. It remains to check that the map is a homomorphism
$\Sect(G,\A)^\op\to \Sect(G^\op,\A^\op)$. But, remembering that the left Haar system
$\lambda^\op$ on $G^\op$ is the right Haar system $(\lambda_x)_{x\in G^0}$ on $G$, we get
\begin{multline*}
\xi^\op*\eta^\op(g)=\int_{G^\op}\xi(h)\cdot_\op\eta(h^{-1}g)\dd(\lambda^\op)^{\rg^\op(g)}(h)\\
=\int_G\eta(gh^{-1})\xi(h)\dd\lambda_{\s(g)}(h)=\int_G\eta(gh)\xi(h^{-1})\dd\lambda^{\rg(g)}(h)=(\eta*\xi)(g)
\end{multline*}
for all $\xi,\eta\in \Sect(G,\A)$ and $g\in G$. This shows that the identity map is an
anti-homomorphism $\Sect(G,\A)\to \Sect(G^\op,\A^\op)$, that is, a homomorphism
$\Sect(G,\A)^\op\to \Sect(G^\op,\A^\op)$, as desired. A similar computation shows that
$\|\xi^\op\|_I=\|\xi\|_I$ and that therefore the identity map extends to an isomorphism
$L^1_I(G,\A)^\op\congto L^1_I(G^\op,\A^\op)$ and hence also to the corresponding
universal enveloping \cstar{}algebras $C^*(G,\A)^\op\congto C^*(G^\op,\A^\op)$.

Finally, to check that $C^*_\red(G,\A)^\op\cong C^*_\red(G^\op,\A^\op)$, fix $x\in G^0$.
The regular representation $\pi^\op_x$ defines a representation of $C^*(G^\op, \A^\op)$
by adjointable operators on the right-Hilbert $\A^\op$-module $L^2(G^\op_x, \A^\op)$. Any
right Hilbert module $\E$ over the opposite $B^\op$ of a \cstar{}algebra $B$ determines a
left Hilbert $B$\nb-module $\E$ with left $B$\nb-action $b\cdot \xi\defeq \xi\cdot b$ and
left $B$\nb-valued inner product $_B\braket{\xi}{\eta}\defeq \braket{\eta}{\xi}_{B^\op}$.
This process preserves the \cstar{}algebras of adjointable operators, meaning that the
identity map on $\E$ yields an isomorphism $\Bound(\E_{B^\op})\cong \Bound(_B\E)$.
Applying this to the right Hilbert $\A_x^\op$\nb-module $L^2(G^\op_x,\A^\op)$ we get a
left Hilbert $\A_x$-module with left $\A_x$-action given by $a\cdot \xi=\xi\cdot_\op a$
for all $\xi\in L^2(G^\op_x,\A^\op)$; the right hand side denotes the right
$\A^\op_x$-action on $L^2(G^\op_x,\A^\op)$, so it is given by $(a\cdot_\op
\xi)(g)=\xi(g)\cdot_\op a(\s^\op(g))=a(\rg(g))\cdot \xi(g)$. The left $\A_x$-valued inner
product on $L^2(G^\op,\A^\op)$ is given by
\[
_{\A_x}\braket{\xi}{\eta}
    = \braket{\eta}{\xi}_{\A_x^\op}
    = \int_{G^\op}\eta(h)^*\cdot_\op \xi(h)\dd\lambda^\op_x(h)
    = \int_G\xi(h)\eta(h)^*\dd\lambda^x(h)
\]
for all $\xi,\eta\in \Sect(G_x^\op,\A^\op)=\Sect(G^x,\A)$. Therefore the left Hilbert
$\A_x$-module obtained from the right Hilbert $\A_x^\op$-module $L^2(G^\op_x,\A^\op)$ in
this way equals the left Hilbert $\A_x$-module $L^2(G^x,\A)$ defined as the completion of
$\Sect(G^x,\A)$ with respect to the norm associated to the left $\A_x$-valued inner
product given by the above formula and the left $\A_x$-action also defined above.
Therefore we may view $\pi_x^\op$ as a representation of $C^*(G^\op,\A^\op)$ on
$\Bound(_{\A_x}L^2(G^x,\A))\cong \Bound(L^2(G^\op_x,\A^\op)_{\A^\op_x})$. Under the
isomorphism $C^*(G^\op,\A^\op)\cong C^*(G,\A)^\op$, this corresponds to the canonical
representation $\tilde\pi_x$ of $C^*(G,\A)^\op$ on $_{\A_x}L^2(G^x,\A)$ via the formula
\[
\tilde\pi_x(\xi)\eta(g)\defeq (\eta*\xi)(g)=\int_G \eta(h)\xi(h^{-1}g)\dd\lambda^x(h)
\]
for $\xi\in \Sect(G,\A)$, $\eta\in \Sect(G^x,\A)$ and $g\in G^x$. Straightforward
computations show that the above formula defines a representation $\tilde\pi_x\colon
C^*(G,\A)^\op\to \Bound(_{\A_x}L^2(G^x,\A))$ of the opposite \cstar{}algebra
$C^*(G,\A)^\op$.

Given a left Hilbert $B$-module $E$, let $\tilde\E$ denote the dual right Hilbert
$B$-module of $\E$: as a vector space $\tilde\E=\{\tilde\xi: \xi\in \E\}$ is the
conjugate of $\E$ and the right $B$-action and right $B$-valued inner product are defined
by $\tilde\xi\cdot b\defeq \widetilde{b^*\cdot \xi}$ and $\braket{\xi}{\eta}_B\defeq
_B\braket{\xi}{\eta}$.

Then each representation $\pi\colon A^\op\to \Bound(_B\E)$ of an opposite \cstar{}algebra
$A^\op$ on the \cstar{}algebra of adjointable operators $\Bound(_B\E)$ of a left Hilbert
$B$-module $\E$ induces a representation $\pi^\op\colon A\to \Bound(_B\E)^\op\cong
\Bound(\tilde\E_B)$. The isomorphism $\Bound(_B\E)^\op\cong \Bound(\tilde\E_B)$ we used
above is induced by the involution, that is, it sends an operator $T\in \Bound(_B\E)$ to
$\tilde T\in \Bound(\tilde\E_B)$ defined by $\tilde T(\tilde\xi)\defeq
(T^*(\xi))^{\sim}$.

For $\xi \in \Sect(G^x,\A)$, the formula $\xi^*(g)\defeq \xi(g^{-1})^*$ determines an
element $\xi^*\in \Sect(G_x,\A)$. The map $\xi \mapsto \xi^*$ induces an isomorphism
$(L^2(G^x,\A))^{\sim}_{\A_x}\cong L^2(G_x,\A)_{\A_x}$ from the dual Hilbert $\A_x$-module
of $_{\A_x}L^2(G^x,\A)$ to the right Hilbert $\A_x$-module $L^2(G_x,\A)$ that carries the
regular representation $\pi_x\colon C^*(G,\A)\to \Bound(L^2(G_x,\A)_{\A_x})$. This
isomorphism intertwines the representations $\pi_x\colon C^*(G,\A)\to
\Bound(L^2(G_x,\A)_{\A_x})$ and $\tilde\pi_x^\op\colon C^*(G,\A)\to
\Bound(_{\A_x}L^2(G^x,\A))^\op\cong \Bound((L^2(G^x,\A)))^{\sim}_{\A_x})$. We conclude
that
\[\|\pi_x^\op(\xi)\|=\|\tilde\pi_x(\xi^\oo)\|=\|\tilde\pi_x^\op(\xi^\oo)\|=\|\pi_x(\xi^\oo)\|.\]
Since $x\in G^0$ was arbitrary, we get the equality $\|\xi^\oo\|_\red=\|\xi\|_\red$ and
therefore the desired isomorphism $C^*_\red(G^\op,\A^\op)\cong C^*_\red(G,\A)^\op$.
\end{proof}

\begin{remark}\label{rmk-conj-bundle}
As in the case of groupoid \cstar{}algebras, we can rephrase the preceding result in
terms of conjugate bundles as well. Let $\A$ be a Fell bundle over a groupoid $G$. For $g
\in G$, let $\overline{\A}_g$ be the conjugate vector space of $\A_g$; that is,
$\overline{\A}_g$ is a copy $\{\overline{a} : a \in \A_g\}$ as an abelian group under
addition, but with scalar multiplication given by $\lambda \overline{a} =
\overline{\bar{\lambda} a}$. Via the map $a \mapsto \bar{a}$, the operations on the Fell
bundle $\A$ induce operations on $\overline{A} := \bigsqcup_{g \in G} \overline{A}_g$:
$\overline{a}\overline{b} = \overline{ab}$, and $\overline{a}^* = \overline{a^*}$. Under
these operations, $\overline{\A}$ is a Fell bundle over $G$, called the \emph{conjugate
bundle} of $\A$.

Let $\A^\oo$ be the opposite bundle of $\A$ defined above; so $\A^\oo_g = \A_{g^{-1}}$.
Then the maps $A^\oo_g \owns a \mapsto \overline{a^*} \in \overline{A}_g$ are linear
isometries because the maps $a \mapsto a^*$ and $a \mapsto \bar{a}$ are both conjugate
linear. We have $\overline{(a \cdot_\op b)^*} = \overline{(ba)^*} = \overline{a^*b^*} =
\overline{a^*}\overline{b^*}$ and $(\overline{a^*}))^* = \overline{a} =
\overline{(a^*)^*}$, so $a \mapsto \overline{a^*}$ determines an isomorphism $\A^\oo
\cong \overline{A}$ of Fell bundles over $G$. Thus Theorem~\ref{thm-fell-op-iso} shows
that there is a topological-\Star{}algebra isomorphism $\xi \mapsto \overline{\xi}$ from
$\Sect(G, \A)$ to $\Sect(G, \overline{A})$ given by $\overline{\xi}(g) \defeq
\overline{\xi(g^{-1})^*}$ that extends to isomorphisms
\[
L^1_I(G, \A) \cong L^1_I(G, \overline{A}),\quad
    C^*(G, \A) \cong C^*(G, \overline{A}),\quad\text{and}\quad
    C^*_{\red}(G, \A) \cong C^*_{\red}(G, \overline{A}).
\]
\end{remark}

\begin{remark}
Remark~\ref{rmk-conj-bundle} is closely related to the idea behind Phillips' construction
in \cite{Phillips:Continuous-trace} of non-self-opposite continuous-trace
\cstar{}algebras $A$; the observation underlying his construction is that the
Dixmier--Douady class of the opposite algebra $A^\op$ is the inverse of the
Dixmier--Douady class of $A$. To see how this relates to our results, fix a compact
Hausdorff space $X$, and let $\Ss$ denote the sheaf of germs of continuous
$\Torus$-valued functions on $X$. The Raeburn--Taylor construction
\cite{Raeburn-Taylor:Continuous-trace} shows that (after identifying $\check{H}^3(X, \Z)$
with $H^2(X, \Ss)$) any class $\delta \in H^2(X, \Ss)$ can be realised as the
Dixmier--Douady invariant of a twisted groupoid \cstar{}algebra $C^*(G, \sigma)$
associated to a continuous 2-cocycle $\sigma$ on a principal \'etale groupoid $G$ with
unit space $G^{(0)} = \bigsqcup_{i,j} U_{ij}$ for some precompact open cover $\{U_i\}$ of
$X$. The cocycle $\sigma$ determines, and is determined up to cohomology by, the Fell
line-bundle $L_\sigma$ over $G$ given by $L_\sigma = G \times \Torus$ with twisted
multiplication $(g,w)(h,z) = (gh, c(g,h)wz)$ and the obvious involution.
Remark~\ref{rmk-conj-bundle} shows that $C^*(G, \sigma)^\op$ is given by the conjugate
bundle $\overline{L_\sigma}$, so the corresponding class in $H^2(X, \Ss)$ is determined
by the pointwise conjugate of the class $\delta$; that is, $\delta(C^*(G, \sigma)) =
\delta(C^*(G, \sigma)^\op)^{-1}$.
\end{remark}

\vskip 0,5 pc

\end{document}